\documentclass[a4paper,12pt]{amsart}
\usepackage{latexsym,amsmath,amssymb}

\newtheorem{thm}{Theorem}[section]

\newtheorem{lem}[thm]{Lemma}
\newtheorem{prop}[thm]{Proposition}
\theoremstyle{definition}
\newtheorem{defn}[thm]{Definition}
\theoremstyle{remark}

\numberwithin{equation}{section}

\begin{document}
\title[]{\textsc{Composition operators on the spaces of harmonic Bloch functions }}
\author{\textsc{S. Esmaeili, Y. Estaremi and A. Ebadian }}
\address{\textsc{S. Esmaeili, Y. estaremi and A. Ebadian}} \email{dr.somaye.esmaili@gmail.com},\email{estaremi@gmail.com},\email{ ebadian.ali@gmail.com.}

\address{Department of mathematics, Payame Noor university , P. O. Box: 19395-3697, Tehran,
Iran.\\
}

\thanks{}
\thanks{}
\date{}
\dedicatory{}

\begin{abstract}
In this paper we characterize some basic properties of composition operators on the spaces of harmonic Bloch functions. First we provide some equivalent conditions for boundedness and compactness of composition operators. In the sequel we investigate closed range composition operators. These results extends the similar results that were proven for composition operators on the Bloch spaces.
\end{abstract}

\maketitle

\commby{}

%%% ----------------------------------------------------------------------
\section{\textsc{Introduction}}

Let $D$ be the open unit disk in the complex plane. For a continuously differentiable complex-valued $f(z)=u(z)+i\upsilon(z),$ $z=x+iy,$ we use the common notation for its formal derivatives:

$$f_{z}=\frac{1}{2}(f_{x}-if_{y}),$$  $$f_{\bar{z}}=\frac{1}{2}(f_{x}+if_{y}).$$

 A twice continuously differentiable complex-valued function $f=u+i\upsilon$ on $D$ is called a \it{harmonic function} if and only if the real-valued function $u$ and $\upsilon$ satisfy Laplace's equation $\Delta u=\Delta \upsilon=0$.\\
A direct calculation shows that the Laplacian of $f$ is
$$\Delta f=4f_{z\bar{z}}.$$
Thus for functions $f$ with continuous second partial derivatives, it is clear that $f$ is harmonic if ana only if $\Delta f=0.$
We consider complex-valued harmonic function $f$ defined in a simply connected domain $D\subset C.$ The function $f$ has a canonical decomposition $f=h+\bar{g},$ where $h$ and $g$ are analytic in $D$ \cite{dp}.
A planar complex-valued harmonic function $f$ in $D$ is called a harmonic Bloch function if and only if

$$\beta_{f}=\sup_{z,w\in D,z\neq w}\frac{|f(z)-f(w)|}{\varrho(z,w)}<\infty.$$

Here $\beta_{f}$ is the Lipschitz number of $f$ and

\begin{align*}
\varrho(z,w)=\arctan h|\frac{z-w}{1-\bar{z}w}|,
\end{align*}

denotes the hyperbolic distance between $z$ and $w$ in $D$, where here $\rho(z,w)$ is the pseudo-hyperbolic distance on $D$.
In \cite{cf} Colonna proved that
\begin{align*}
\beta_{f}=\sup_{z\in D}(1-|z|^2)[|f_{z}(z)|+|f_{\bar{z}}(z)|].
\end{align*}
Moreover, the set of all harmonic Bloch mappings, denoted by the symbol $HB(1)$ or $HB$, forms a complex Banach space with the norm $\|.\|$ given by
$$\|f\|_{HB(1)}=|f(0)|+\sup_{z\in D}(1-|z|^2)[|f_{z}(z)|+|f_{\bar{z}}(z)|].$$
 \begin{defn} For $\alpha\in(0,\infty)$, the \it{harmonic $\alpha$-Bloch} space $HB(\alpha)$ consists of complex-valued harmonic function $f$ defined on $D$ such that
$$|||f|||_{HB(\alpha)}=\sup_{z\in D}(1-|z|^2)^\alpha[|f_{z}(z)|+|f_{\bar{z}}(z)|]<\infty,$$
and the harmonic little $\alpha$-Bloch space $HB_{0}(\alpha)$ consists of all function in $HB(\alpha)$ such that
$$\lim_{|z|\rightarrow1}(1-|z|^2)^\alpha[|f_{z}(z)|+|f_{\bar{z}}(z)|]=0.$$
\end{defn}
Obviously, when $\alpha=1$, we have $|||f|||_{HB(\alpha)}=\beta_{f}$. Each $HB(\alpha)$ is a Banach space with the norm given by
\begin{align*}
\|f\|_{HB(\alpha)}&=|f(0)|+\sup_{z\in D}(1-|z|^2)^\alpha[|f_{z}(z)|+|f_{\bar{z}}(z)|],
\end{align*}
and $HB_{0}(\alpha)$ is a closed subspace of $HB(\alpha)$.
Now we define composition operators.
 \begin{defn}Let $D$ be the open unit disk in the complex plane. Let $\varphi$ be an analytic self-map of $D$, i. e., an analytic function $\varphi$ in $D$ such that $\varphi(D)\subset D$. The composition operator $C_{\varphi}$ induced by such $\varphi$ is the linear map on the spaces of all harmonic functions on the unit disk defined by
$$C_{\varphi}f=fo\varphi.$$
\end{defn}
 The composition operators on function spaces were studied by many authors. Some known results about composition operators can be found in \cite{cm} and \cite{sh}.
In this paper we study composition operators on harmonic Bloch-type spaces $HB(\alpha)$.
In section 2, by using of Theorem 2.1 in \cite{lz}, we give a necessary and sufficient condition for the boundedness of $C_{\varphi}$ on $HB(\alpha)$ for $\alpha\in(0,\infty)$, which extends Theorem 3.1 in \cite{lz}, by Lou.
The compactness of $C_{\varphi}$ on analytic Bloch-type spaces were characterized in\cite{ka,lz}. In this paper, we deal the compactness of composition operators between the Banach spaces of harmonic function $HB(\alpha)$ and $HB_{0}(\alpha)$. 
 
Moreover, we investigate closed range composition operators. Closed range composition operators on the Bloch-type spaces have been studied in \cite{ch,cg,gzz,zon}).
The isometric composition operators on Bloch-type spaces have been studied in a number of papers (such as \cite{cco,cf,xc,zn}).
For $\alpha>0$, and $\varphi$ being an analytic self-map of $D$, let
$$\tau_{\varphi,\alpha}(z)=\frac{(1-|z|^2)^\alpha|\varphi^{'}(z)|}{(1-|\varphi(z)|^2)^\alpha}.$$
We write $\tau_{\varphi}$ if $\alpha=1$.
We say that a subset $G\subset D$ is called sampling set for $HB(\alpha)$ if $\exists S>0$ such that for all $f\in HB(\alpha)$,
$$\sup _{z\in G}(1-|z|^2)^\alpha [|f_{z}(z)|]+[|f_{\bar{z}}(z)|]\geq S |||f|||_{HB(\alpha)}.$$
To state the results obtained, we need the following definition.
Let $\rho(z,w)=|\varphi_{z}(w)|$ denote the pseudohyperbolic distance (between $z$ and $w$) on $D$, where $\varphi_{z}$ is a disk automorphism of $D$ that is

$$\varphi_{z}(w)=\frac{z-w}{1-\bar{z}w}.$$
We say that subset $G\subset D$ is an $r$-net for $D$ for some $r\in(0,1)$ if for each $z\in D$, $\exists w\in G$ such that $\rho(z,w)<r$. For $c>0$, let
$$\Omega_{c,\alpha}=\{z\in D: \tau_{\varphi,\alpha}(z)\geq c\},$$
and let $G_{c,\alpha}=\varphi(\Omega_{c,\alpha})$. If $\alpha=1$, we write $\Omega_{c}$ and $G_{c}$.
Now we recall Montel's theorem for harmonic functions.
 \begin{thm}\cite{abr} If $\{u_{n}\}_{n=1}^{\infty}$ is a sequence of harmonic functions in the region $\Omega$ with
$\sup_{n,x\in K}|u_{n}(x)|<\infty$ for every compact set $K\subset\Omega$, then there exists a subsequence, $\{u_{n_{j}}\}_{j=1}^{\infty}$ converging uniformly on every compact set $K\subset\Omega$.
\end{thm}
Also we recall a very useful theorem that we will use it a lot in this paper.
\begin{thm}\label{t1}\cite{lz}
Let $0<\alpha<\infty$. Then there exist $f,g\in HB(\alpha)$ such that

$$|f'(z)|+|g'(z)|\geq\frac{1}{(1-|z|)^\alpha},$$
for all $z\in D$.
\end{thm}
\section{\textsc{Main results}}
In this section we study bounded and compact composition operators on $HB(\alpha)$. And then we nvestigate closed range composition operators on $HB(\alpha)$. First we provide some equivalent conditions for boundedness of composition operator $C_{\varphi}$ on $HB(\alpha)$.

\begin{thm}\label{t1}
 If $0 < \alpha < \infty$, $\varphi \in H(D)$ and $\varphi(D)\subseteq D$, then the following statements are equivalent:\\
a) $C_{\varphi}:HB(\alpha)\rightarrow HB(\alpha)$ is bounded.\\
b)
$$\sup _{z\in D}\frac{(1-|z|^2)^\alpha}{(1-|\varphi(z)|^2)^\alpha}|\varphi^{'}(z)|<\infty.$$
\end{thm}
\begin{proof}
For the implication $a\rightarrow b$, by Theorem 2.1 of \cite{lz} we have that for $0 < \alpha < \infty$ there exist $h,g\in B(\alpha)$ satisfying the inequality
$$|h^{'}(z)|+|g^{'}(z)|\geq \frac{1}{(1-|z|)^\alpha}.$$

If we set $f=h+\bar{g}\in HB(\alpha)$, then $fo\varphi (z)=ho\varphi(z)+\overline{go\varphi (z)}$ and so by the same method of Theorem 3.1 of \cite{lz} we get the proof.\\
For the implication $b\rightarrow a$ we can do the same as Theorem 3.1 of \cite{lz}.
\end{proof}
In the next theorem we consider the composition operator from $HB_{0}(\alpha)$ into $HB(\alpha)$ and we find some conditions under which $C_{\varphi}$ is bounded.
\begin{thm}\label{t2} Let $0<\alpha<\infty$, $\varphi \in H(D)$ and $\varphi(D)\subseteq D$. Then the followings are equivalent:\\
a) $C_{\varphi}:HB_{0}(\alpha)\rightarrow HB(\alpha)$ is bounded.\\
b)
$$\sup _{z\in D}\frac{(1-|z|^2)^\alpha}{(1-|\varphi(z)|^2)^\alpha}|\varphi^{'}(z)|<\infty.$$
\end{thm}
\begin{proof} The proof is similar to the proof of Theorem 3.3 of \cite{lz}. Hence we omit the proof.
\end{proof}
Now we consider the composition operator $C_{\varphi}:HB(\alpha)\rightarrow HB_{0}(\alpha)$ and we give an equivalent condition to boundedness of $C_{\varphi}$.
\begin{thm}\label{t3} If $0 < \alpha < \infty$, $\varphi \in H(D)$ and $\varphi(D)\subseteq D$, then the following are equivalent:\\

a) $C_{\varphi}:HB(\alpha)\rightarrow HB_{0}(\alpha)$ is bounded.\\

b)
$$\lim _{|z|\rightarrow 1}\frac{(1-|z|^2)^\alpha}{(1-|\varphi(z)|^2)^\alpha}|\varphi^{'}(z)|=0.$$
\end{thm}
\begin{proof} By a similar method of the proof of Theorem 3.4 of \cite{lz} we get the proof.
\end{proof}
Finally we provide some conditions for boundedness of the composition operator $C_{\varphi}$ as an operator on $HB_{0}(\alpha)$.
\begin{thm}\label{t4} If $0 < \alpha < \infty$, $\varphi \in H(D)$ and $\varphi(D)\subseteq D$, then the following are equivalent:\\
a) $C_{\varphi}:HB_{0}(\alpha)\rightarrow HB_{0}(\alpha)$ is bounded.\\

b) $\varphi \in B_{0}( \alpha )$ and
$$\sup _{z\in D}\frac{(1-|z|^2)^\alpha}{(1-|\varphi(z)|^2)^\alpha}|\varphi^{'}(z)|<\infty.$$
\end{thm}
\begin{proof} By some simple calculations one can get the proof.
\end{proof}
A sequence $\{z_{n}\}$ in $D$ is said to be $R$-separated if $\rho(z_{n},z_{m})=|\frac{z_{m}-z_{n}}{1-\bar{z_{m}}z_{n}}|>R$ whenever $m\neq n$. Thus an $R$-separated sequence consists of points which are uniformly far apart in the  pseudohyperbolic metric on $D$, or equivalently, the hyperbolic balls $D(z_{n},r)=\{w: \rho(w,z_{n})<r\}$ are pairwise disjoint for some $r>0$. Evidently, any sequence $\{z_{n}\}$  in $D$ which satisfies $|z_{n}|\rightarrow1$ possesses an $R$-separated subsequence for any $R>0$.

Another property of separated sequence is contained in  the next proposition.
\begin{prop}\label{p0}\cite{ka}.
There is an absolute constant $R>0$ such that if $\{z_{n}\}$ is $R$-separated, then for every bounded sequence $\{\lambda_{n}\}$ there is an $f\in B$ such that $(1-|z_{n}|^2)f'(z_{n})=\lambda_{n}$ for all $n$.
\end{prop}

Since every sequence $\{z_{n}\}$ with $|z_{n}|\rightarrow1$ contains an $R$-separated subsequence $\{z_{n_{k}}\}$, it follows that there is an $f\in B$ such that $(1-|z_{n_{k}}|^2)f'(z_{n_{k}})=1$ for all $k$.\\
Now we begin investigating compactness of the composition operator $C_{\varphi}$ in different cases. First we provide some equivalent conditions for compactness of $C_{\varphi}$ as an operator on $HB(\alpha)$.
\begin{thm}\label{t5} Let $0 < \alpha < \infty$,  $\varphi \in H(D)$ and $\varphi(D)\subseteq D$. Then we have the followings equivalent conditions:\\
a) $C_{\varphi}:HB(\alpha)\rightarrow HB(\alpha)$ is compact.\\
b) $$\lim _{|\varphi(z)|\rightarrow 1}\left(\frac{1-|z|^2}{1-|\varphi(z)|^2}\right)^\alpha|\varphi^{'}(z)|=0,$$
and
 $$\sup _{z\in D}\left(\frac{1-|z|^2}{1-|\varphi(z)|^2}\right)^\alpha|\varphi^{'}(z)|<\infty.$$
\end{thm}
\begin{proof}
By making use of the proof of Theorem 4.2 of \cite{lz} and the Proposition 1 of \cite{ka} we get the proof of Proposition 1 of \cite{ka}

\end{proof}
Here we prove that the compactness of $C_{\varphi}:HB_{0}(\alpha)\rightarrow HB_{0}(\alpha)$ and $C_{\varphi}:HB(\alpha)\rightarrow HB_{0}(\alpha)$  are equivalent and we find an equivalent condition for compacness of $C_{\varphi}$ in these cases.
\begin{thm}\label{t6}
Let $0 < \alpha < \infty$, $\varphi \in H(D)$ and $\varphi(D)\subseteq D$. Then the following statements are equivalent:

a) The operator  $C_{\varphi}:HB_{0}(\alpha)\rightarrow HB_{0}(\alpha)$ is compact.\\

b) The operator $C_{\varphi}:HB(\alpha)\rightarrow HB_{0}(\alpha)$ is compact.\\

c) $$\lim _{|z|\rightarrow 1}\frac{(1-|z|^2)^\alpha}{(1-|\varphi(z)|^2)^\alpha}|\varphi^{'}(z)|=0.$$
\end{thm}

\begin{proof} First we prove the implication $a\rightarrow c$.  If $C_{\varphi}:HB_{0}(\alpha)\rightarrow HB_{0}(\alpha)$ is compact, then the set
$K=\overline{C_{\varphi}(S_{HB_{0}(\alpha)})}\subset HB_{0}(\alpha)$ compact, in which $S_{HB_{0}(\alpha)}=\{f\in HB_{0}(\alpha):\|f\|_{HB_{0}(\alpha)} \leq 1\}$. By  the Theorem \ref{t5}

we get that
$$\sup_{\|f\|_{HB(\alpha)} \leq1} (1-|z|^2)^\alpha [|f_{z}(z)|+|f_{\bar{z}}(z)|]=1$$
for all $z\in D$. Moreover we have
\begin{align*}
0&=\lim_{|z|\rightarrow 1} \sup_{\|f\|_{HB(\alpha)} \leq1} (1-|z|^2)^\alpha [|(fo\varphi)_{z}(z)|+|(fo\varphi)_{\bar{z}}(z)|]\\
&=\lim _{|z|\rightarrow 1}\frac{(1-|z|^2)^\alpha}{(1-|\varphi(z)|^2)^\alpha}|\varphi^{'}(z)| \sup_{\|f\|_{HB(\alpha)} \leq 1}(1-|\varphi(z)|^2)^\alpha [|h^{'}(\varphi(z)|)+|g^{'}(\varphi(z))|].
\end{align*}
So we get the desired result.\\
Now we prove the implication $c\rightarrow b$. Let $\{f_{n}\}_{n\in \mathbb{N}}\subset HB(\alpha)$ and $\|f_{n}\|_{HB(\alpha)}\leq 1$, for all $n$. First we obtain that $\{C_{\varphi}f_{n}\}$ has a subsequence that converges in $HB_{0}(\alpha)$. By Montel's Theorem we have a subsequence $\{f_{n_{k}}\}\subset \{f_{n}\}$, that converges uniformly on subsets of $D$ to a harmonic function $f$. Hence we have
\begin{align*}
(1-|z|^2)^\alpha [|f_{z}(z)|+|f_{\bar{z}}(z)|]&=\lim_{k\rightarrow \infty}(1-|z|^2)^\alpha [|(f_{n_{k}})_{z}(z)|+|(f_{n_{k}})_{\bar{z}}(z)|]\\
&\leq\lim_{k\rightarrow \infty} \|f_{n_{k}}\|_{HB(\alpha)}\\
& \leq1.
\end{align*}
This means that $f\in HB(\alpha)$ with $\|f\|_{HB(\alpha)}\leq 1$. Also we have
\begin{align*}
(1-|z|^2)^\alpha [|(fo\varphi)_{z}(z)|+|(fo\varphi)_{\bar{z}}(z)|]&=\frac{(1-|z|^2)^\alpha}{(1-|\varphi(z)|)^\alpha}|\varphi^{'}(z)|\\
&\leq \frac{(1-|z|^2)^\alpha}{(1-|\varphi(z)|^2)^\alpha} |\varphi^{'}(z)| \|f\|_{HB(\alpha)}.
\end{align*}

By these observations we conclude that $C_{\varphi}f\in HB_{0}(\alpha)$. Also we need to show that
$$\lim_{k\rightarrow \infty}\|C_{\varphi}f_{n_{k}}-C_{\varphi}f\|_{HB(\alpha)}=0.$$
Since  $\lim _{|z|\rightarrow 1}\frac{(1-|z|^2)^\alpha}{(1-|\varphi(z)|^2)^\alpha}|\varphi^{'}(z)|=0$, then for any $\varepsilon>0$, there exists $r\in (0,1)$ such that for $z$ with $r<|z|<1$ we have
$$ \frac{(1-|z|^2)^\alpha}{(1-|\varphi(z)|^2)^\alpha} |\varphi^{'}(z)|<\frac{\varepsilon}{4}.$$
And so for all $z$ with $r<|z|<1$ we have
\begin{align*}
(1-|z|^2)^\alpha |((f_{n_{k}}-f)\circ\varphi)'(z)|&=(1-|z|^2)^\alpha\{[|(f_{n_{k}})_{z}\varphi(z)|+|(f_{n_{k}})_{\bar{z}}\varphi(z)|]\}\\
&-(1-|z|^2)^\alpha\{[|f_{z}\varphi(z)|+|f_{\bar{z}}\varphi(z)|]\}\\
&\leq\frac{\varepsilon}{4} (\|f_{n_{k}}\|_{HB(\alpha)}+\|f\|_{HB(\alpha)})\leq \frac{\varepsilon}{2}.
\end{align*}
For $z$ with $|z|\leq r$, the set $\{\varphi(z): |z|\leq r\}$ is a compact subset of $D$. Since
$$(1-|z|^2)^\alpha [|f_{z}(z)|+|f_{\bar{z}}(z)|]=\lim_{k\rightarrow \infty}(1-|z|^2)^\alpha
[|(f_{n_{k}})_{z}(z)|+|(f_{n_{k}})_{\bar{z}}(z)|]$$
and
\begin{align*}
(1-|z|^2)^\alpha |((f_{n_{k}}-f)\circ\varphi)'(z)|&\leq (1-|z|^2)^\alpha \{[|(f_{n_{k}})_{z}\varphi(z)|+|(f_{n_{k}})_{\bar{z}}\varphi(z)|]\\
&-[|f_{z}\varphi(z)|+|f_{\bar{z}}\varphi(z)|]\times\sup _{z\in D}\frac{(1-|z|^2)^\alpha}{(1-|\varphi(z)|^{2})^\alpha}|\varphi^{'}(z)|.
\end{align*}
Hence we have
$(1-|z|^2)^\alpha |((f_{n_{k}}-f)o\varphi)'(z)|\rightarrow0$ uniformly on $\{z:|z|\leq r\}$. Therefore  $(1-|z|^2)^\alpha |((f_{n_{k}}-f)o\varphi)'(z)|<\frac{\varepsilon}{2}$ for $k$ sufficiently large and $\{z:|z|\leq r\}$. This completes the proof.\\
The implication $b\rightarrow a$ is clear.
\end{proof}
Let $(X,d)$ be a metric space and let $\varepsilon>0$. We say that $A\subset X$ is an $\varepsilon$-net for $(X,d)$, if for all $x\in X$ there exists a $a$  in $A$ such that $d(a,x)<\varepsilon$. We characterize the compact subsets of $HB_{0}(\alpha)$ in the next lemma.
\begin{lem}\label{t7}
A closed subset of $HB_{0}(\alpha)$ is compact if and only if it is bounded and satisfies
$$\lim_{|z|\rightarrow1}\sup_{f\in k}(1-|z|^2)^\alpha [|f_{z}(z)|+|f_{\bar{z}}(z)|]=0.$$
\end{lem}
\begin{proof}
suppose that $K\subset HB_{0}(\alpha)$ is compact and $\varepsilon >0$. Then we can choose an $\frac{\varepsilon}{2}$-net $f_{1},f_{2},...,f_{n}\in K$. hence there exists $\delta$, $0<\delta<1$, such that for all $z$ with $|z|>\delta$ we have $(1-|z|^2)^\alpha [|(f_{i})_{z}(z)|+|(f_{i})_{\bar{z}}(z)|]<\frac{\varepsilon}{2}$ for all $1\leq i\leq n.$ If $f\in K$, then there exists some  $f_{i}$ such that $\|f-f_{i}\|_{HB(\alpha)}<\frac{\varepsilon}{2}$ and so for all $z$ with $|z|>\delta$ we have
$$(1-|z|^2)^\alpha [|f_{z}(z)|+|f_{\bar{z}}(z)|]\leq \|f-f_{i}\|_{HB(\alpha)}+(1-|z|^2)^\alpha [|(f_{i})_{z}(z)|+|(f_{i})_{\bar{z}}(z)|]<\varepsilon.$$
Therefor we get that
$$\lim_{|z|\rightarrow1}\sup_{f\in k}(1-|z|^2)^\alpha [|f_{z}(z)|+|f_{\bar{z}}(z)|]=0.$$
Conversely, let $K$ be a closed and bounded subset of $HB_{0}(\alpha)$ such that
 $$\lim_{|z|\rightarrow1}\sup_{f\in k}(1-|z|^2)^\alpha [|f_{z}(z)|+|f_{\bar{z}}(z)|]=0.$$

  Since $K$ is bounded, then it is relatively compact with respect to the topology of the uniform convergence on compact subsets of the unit disk. If $(f_{n})$ is a sequence in $K$, then by Montel's Theorem we have a subsequence $\{f_{n_{k}}\}\subset \{f_{n}\}$ which converges uniformly on compact subsets of $D$ to a harmonic function $f$. Also $\{f_{n_{k}}'\}$ converges uniformly to $f'$ on compact subsets of $D$. For every $\varepsilon>0$
  %Since $f_{n_{k}}$ is contained in $K$,
  we can find $\delta>0$ such that for all $z$ with $|z|>\delta$ we have  $$(1-|z|^2)^\alpha[|(f_{n_{k}})_{z}(z)|+|(f_{n_{k}})_{\bar{z}}(z)|]<\frac{\varepsilon}{2}$$
   for any integer $k>0$. Therefor
$(1-|z|^2)^\alpha [|f_{z}(z)|+|f_{\bar{z}}(z)|]<\frac{\varepsilon}{2},$
for all $z$ with $|z|>\delta$. So
\begin{align*}
\sup_{|z|>\delta}(1-|z|^2)^\alpha [|(f_{n_{k}}-f)_{z}(z)|+|(f_{n_{k}}-f)_{\bar{z}}(z)|]&\leq\sup_{|z|>\delta}(1-|z|^2)^\alpha[|(f_{n_{k}})_{z}(z)|+|(f_{n_{k}})_{\bar{z}}(z)|]\\
&+\sup_{|z|>\delta}(1-|z|^2)^\alpha[|f_{z}(z)|+|f_{\bar{z}}(z)|]\\
&<\varepsilon.
\end{align*}
Moreover, since $(f_{n_{k}})$ converges uniformly on compact subsets of $D$ to $f$ and $(f_{n_{k}}')$ converges uniformly to $f'$ on $\{z:|z|\leq \delta\}$, we get that
$$\sup_{|z|\leq \delta}(1-|z|^2)^\alpha [|(f_{n_{k}}-f)_{z}(z)|+|f_{n_{k}}-f)_{\bar{z}}(z)|]\leq\varepsilon.$$
Consequently for $k$ large enough, we have $\lim_{k\rightarrow\infty}\|f_{n_{k}}-f\|_{HB(\alpha)}\leq \varepsilon$. This completes the proof.
%Since $\varepsilon>0$, is arbitrary, $\lim_{k\rightarrow\infty}\|f_{n_{k}}-f\|_{HB(\alpha)}=0$ and so $K$ is compact.\cite{}.
\end{proof}
In the next theorem we prove that the norm convergence in $HB(\alpha)$ implies the uniform convergence.

\begin{thm}\label{t8}
The norm convergence in $HB(\alpha)$ implies the uniform convergence, that is if $\{f_{n}\}\subset HB(\alpha)$ such that $\|f_{n}-f\|_{HB(\alpha)}\rightarrow0$, then $\{f_{n}\}$ converges uniformly to $f$.
\end{thm}
\begin{proof} For $0\neq z\in D$, we have
\begin{align*}
|f_{n}(z)-f(z)|&=|\int_{0}^{1}\frac{d(f_{n}-f)}{dt}(zt)dt|\\
&=|z \int_{0}^{1}\frac{d(f_{n}-f)}{d \varsigma(t)}(zt)dt+\bar{z}\int_{0}^{1}\frac{d(f_{n}-f)}{d \bar{\varsigma}(t)}(zt)dt|\\
%&\leq|z| [\int_{0}^{1}|(f_{n}-f)_{\varsigma(t)}(zt)|dt+\int_{0}^{1}|(f_{n}-f)_{\bar {\varsigma(t)}}(zt)|dt]\\
&\leq|z| \int_{0}^{1}[|(f_{n}-f)_{\varsigma(t)}(zt)|+|(f_{n}-f)_{\bar{\varsigma(t)}}(zt)|]dt,
\end{align*}
in which $\varsigma(t)=zt$.
This gives us
\begin{align*}
|f_{n}(z)-f(z)|&\leq\int_{0}^{1}\frac{[|(f_{n}-f)_{\varsigma(t)}(zt)|+|(f_{n}-f)_{\bar{\varsigma(t)}}(zt)|]}{(1-|\varsigma(t)|^{2})^{\alpha}}(1-|\varsigma(t)|^{2})^{\alpha}dt\\
&\leq(\|f_{n}-f\|_{HB(\alpha)})\int_{0}^{1}\frac{1}{(1-|z|t)^{\alpha}}dt \rightarrow 0,
\end{align*}
when $n\rightarrow\infty$. So we get the proof.
\end{proof}
We say that a subset $G\subset D$ is called sampling set for $HB(\alpha)$ if $\exists S>0$ such that for all $f\in HB(\alpha)$,
$$\sup _{z\in G}(1-|z|^2)^\alpha [|f_{z}(z)|]+[|f_{\bar{z}}(z)|]\geq S \|f\|_{HB(\alpha)}.$$
In the next theorem we provide some equivalent conditions for closedness of range of the composition operator on $HB(\alpha)$.
\begin{thm}\label{t9}
Let $\varphi:D\rightarrow D$, $\alpha>0$ and $C_{\varphi}:HB(\alpha)\rightarrow HB(\alpha)$ be a bounded operator. Then the range of
$C_{\varphi}:HB(\alpha)\rightarrow HB(\alpha)$ is closed if and only if there exists $c>0$ such that $G_{c,\alpha}$ is sampling for $HB(\alpha)$.
\end{thm}
\begin{proof}
Since $C_{\varphi}:HB(\alpha)\rightarrow HB(\alpha)$ is  bounded, then $\exists K>0$ such that $\sup_{z\in D}\tau_{\varphi,\alpha}(z)\leq K$. Since every non-constant  $\varphi$ is an open map, then the composition operator $C_{\varphi}$ is always one
to one. By a basic operator theory result, a one-to-one operator has closed range if and only if it is bounded below. hence if $C_{\varphi}$ has closed range, then $C_{\varphi}$ is bounded below, that is $\exists \varepsilon>0$ such that for all $f\in HB(\alpha)$,
\begin{align*}
\|C_{\varphi}f\|_{HB(\alpha)}&=\sup_{z\in D} (1-|z|^2)^\alpha [|(fo\varphi)_{z}(z)|+|(fo\varphi)_{\bar{z}}(z)|]\\
%&=\sup_{z\in D} (1-|z|^2)^\alpha |\varphi^{'}(z)| [|h^{'}(\varphi(z)|)+|g^{'}(\varphi(z))|]\\
&=\sup_{z\in D}\tau_{\varphi,\alpha}(z)|(1-|\varphi(z)|^2)^\alpha [|h^{'}(\varphi(z)|)+|g^{'}(\varphi(z))|]\\
&\geq\varepsilon \|f\|_{HB(\alpha)}.
\end{align*}
Now we show that the set $G_{c,\alpha}$ is sampling for $HB(\alpha)$ with sampling constant $S=\frac{\varepsilon}{K}$. Since $\Omega_{c,\alpha}=\{z\in D : \tau_{\varphi,\alpha}(z)\geq c\}$, so for any $z\notin \Omega_{c,\alpha}$ and $c=\frac{\varepsilon}{2}$, we have
$$\sup_{z\notin \Omega_{c,\alpha}}\tau_{\varphi,\alpha}(z)|(1-|\varphi(z)|^2)^\alpha [|h^{'}(\varphi(z)|)+|g^{'}(\varphi(z))|]\leq\frac{\varepsilon}{2}\|f\|_{HB(\alpha)}.$$
Therefore we have
\begin{align*}
\varepsilon \|f\|_{HB(\alpha)}&\leq \sup_{z\in D}\tau_{\varphi,\alpha}(z)|(1-|\varphi(z)|^2)^\alpha [|h^{'}(\varphi(z)|)+|g^{'}(\varphi(z))|]\\
&=\sup_{z\in \Omega_{c,\alpha}}\tau_{\varphi,\alpha}(z)(1-|\varphi(z)|^2)^\alpha [|h^{'}(\varphi(z)|)+|g^{'}(\varphi(z))|]\\
%&\leq K \sup_{z\in \Omega_{c,\alpha}}(1-|\varphi(z)|^2)^\alpha [|h^{'}(\varphi(z)|)+|g^{'}(\varphi(z))|]\\
&\leq K \sup_{w\in G_{c,\alpha}}(1-|w|^2)^\alpha [|h^{'}(w|)+|g^{'}(w)|].
\end{align*}
Hence $\sup_{w\in G_{c,\alpha}}(1-|w|^2)^\alpha [|h^{'}(w|)+|g^{'}(w)|]\geq \frac{\varepsilon}{K}\|f\|_{HB(\alpha)}$. this means that $G_{c,\alpha}$ is a sampling set for $HB(\alpha)$ with sampling constant $S=\frac{\varepsilon}{K}$.\\
Conversely, suppose that  $G_{c,\alpha}$ is a sampling set for $HB(\alpha)$, with sampling constant $S>0$. So for all $f\in HB(\alpha)$ and $\varepsilon=cS$ we get the followings relations:
\begin{align*}
S\|f\|_{HB(\alpha)}&\leq \sup_{z\in \Omega_{c,\alpha}}(1-|\varphi(z)|^2)^\alpha [|(f)_{z}(\varphi(z))|+|(f)_{\bar{z}}(\varphi(z))|]\\
&=\sup_{z\in \Omega_{c,\alpha}}(1-|\varphi(z)|^2)^\alpha [|h^{'}(\varphi(z)|)+|g^{'}(\varphi(z))|]\\
%&=\sup_{z\in \Omega_{c,\alpha}} \tau_{\varphi,\alpha}^{-1}(z) (1-|z|^2)^\alpha |\varphi^{'}(z)| [|h^{'}(\varphi(z)|)+|g^{'}(\varphi(z))|]\\
&\leq\frac{1}{c} \sup_{z\in D}(1-|z|^2)^\alpha [|(ho\varphi)_{z}(z)|+|(go\varphi)_{\bar{z}}(z)|]\\
%&\leq\frac{1}{c} \sup_{z\in D}(1-|z|^2)^\alpha [|(fo\varphi)_{z}(z)|+|(fo\varphi)_{\bar{z}}(z)|]\\
&\leq\frac{1}{c} \|fo\varphi\|_{HB(\alpha)}.
\end{align*}
Therefore
$$\varepsilon \|f\|_{HB(\alpha)}\leq \|fo\varphi\|_{HB(\alpha)}=\|C_{\varphi}f\|_{HB(\alpha)}.$$
Hence $C_{\varphi}$ is bounded below and so $C_{\varphi}$ has closed range.
\end{proof}
Now we give some other necessary and sufficient conditions for closedness of range of $C_{\varphi}:HB(\alpha)\rightarrow HB(\alpha)$.
\begin{thm}\label{t10} Let $\varphi$ be a self-map of $D$, $\alpha>0$, and $C_{\varphi}:HB(\alpha)\rightarrow HB(\alpha)$ be a bounded operator. Then we have the followings:

a) If the operator $C_{\varphi}:HB(\alpha)\rightarrow HB(\alpha)$ has closed range, then there exist $c,r>0$ with $r<1$, such that $G_{c,\alpha}$ is an r-net for $D$.

b) If there exist $c,r>0$ with $r<1$, such that $G_{c,\alpha}$ contains an open annulus centered at the origin and with outer radius $1$, then $C_{\varphi}$ has closed range.
\end{thm}
\begin{proof} a) For $a\in D,$ let $\varphi_{a}(z)$ be a function such that $\varphi_{a}(0)=0$ and $\varphi_{a}^{'}(z)=(\psi_{a}^{'}(z))^{\alpha}$, where $\psi_{a}$ is the disc automorphism of $D$ defined by $\psi_{a}(z)=\frac{a-z}{1-\bar{a}z}$. Using the equalities
$$1-\rho(z,w)^{2}=1-|\psi_{w}(z)|^{2}=(1-|z|^2)|\psi_{w}^{'}(z)|$$
we get
\begin{align*}
\|\varphi_{a}+\bar{\varphi_{a}}\|_{HB(\alpha)}&=\sup_{z\in D}(1-|z|^2)^\alpha 2|\varphi_{a}^{'}(z)|\\
%&=\sup_{z\in D}(1-|z|^2)^\alpha 2|\psi_{a}^{'}(z)|^{\alpha}\\
&=2\sup_{z\in D} (1-|\psi_{a}(z)|^{2})^{\alpha}=2.
\end{align*}

If we put $f=\varphi_{a}+\bar{\varphi_{a}}$, then we have
\begin{align*}
\|C_{\varphi}f\|_{HB(\alpha)}&=\|fo\varphi\|_{HB(\alpha)}\\
&=\sup_{z\in D} (1-|z|^2)^\alpha [|(fo\varphi)_{z}(z)|+|(fo\varphi)_{\bar{z}}(z)|]\\
&=\sup_{z\in D} \tau_{\varphi,\alpha}(z)2(1-|\psi_{a}(\varphi(z))|^2)^\alpha.
\end{align*}
Moreover, by assuming that $C_{\varphi}$ is bounded and has closed range, then there exist $K,\varepsilon>0$ such that $\sup_{z\in D} \tau_{\varphi,\alpha}(z)= K$ and
\begin{align*}
\|fo\varphi\|_{HB(\alpha)}&=\sup_{z\in D} \tau_{\varphi,\alpha}(z) 2(1-|\psi_{a}(\varphi(z))|^2)^\alpha\\
 &\geq \varepsilon \|\varphi_{a}+\bar{\varphi_{a}}\|_{HB(\alpha)}.
 \end{align*}
 This implies that
 \begin{align*}
 \varepsilon &\leq \sup_{z\in D} \tau_{\varphi,\alpha}(z) (1-|\psi_{a}(\varphi(z))|^2)^\alpha\\
 &\leq \sup_{z\in D} \tau_{\varphi,\alpha}(z)=K.
 \end{align*}

Since $1-|\psi_{a}(\varphi(z))|^2\leq 1$, then there exists $z_{a}\in D$ such that
$$\tau_{\varphi,\alpha}(z_{a})\geq \frac{\varepsilon}{2}$$
and
$$(1-|\psi_{a}(\varphi(z_{a}))|^2)^\alpha \geq \frac{\varepsilon}{2K}.$$
Thus, for $c=\frac{\varepsilon}{2}$ and $r=\sqrt{1-(\frac{\varepsilon}{2K})^{\frac{1}{\alpha}}}$, we conclude that for all $a\in D$, there exists $z_{a}\in \Omega_{c,\alpha}$ such that $\rho(a,\varphi(z_{a}))<r$ and so $G_{c,\alpha}$ is an r-net for $D$.\\
b) Let $G_{c,\alpha}$ contains the annulus $A=\{z: r_{0}<|z|<1\}$ and $C_{\varphi}:HB(\alpha)\rightarrow HB(\alpha)$ be  bounded. Suppose that $C_{\varphi}$ doesn't have closed range, then there exists a sequence $\{f_{n}\}$ with $\|f_{n}\|_{HB(\alpha)}=1$ and $\|C_{\varphi}f_{n}\|_{HB(\alpha)}\rightarrow 0$. For each $\varepsilon>0$, let $N_{\varepsilon}>0$ such that for all $n>N_{\varepsilon}$ we have $$\|C_{\varphi}f_{n}\|_{HB(\alpha)}<\varepsilon<c\varepsilon.$$
Since
$$\sup_{z\in D} (1-|z|^2)^\alpha [|(f_{n})_{z}(z)|+|(f_{n})_{\bar{z}}(z)|]=\sup_{z\in D} (1-|z|^2)^\alpha [|h_{n}^{'}(z)|+|g_{n}^{'}(z)|]=1,$$
 then there exists a sequence $\{a_{n}\}$ in $D$ such that for all $n$,
$$(1-|a_{n}|^2)^\alpha [|h_{n}^{'}(a_{n})|+|g_{n}^{'}(a_{n})|]\geq \frac{1}{2}.$$
Moreover, we have
\begin{align*}
&\sup_{w\in G_{c,\alpha}} (1-|w|^2)^\alpha [|(f_{n})_{z}(w)|+|(f_{n})_{\bar{z}}(w)|]\\
&=\sup_{z\in \Omega_{c,\alpha}}\tau_{\varphi,\alpha}^{-1}(z) \tau_{\varphi,\alpha}(z)(1-|\varphi(z)|^2)^\alpha [|(f_{n})_{z}(\varphi(z))|+|(f_{n})_{\bar{z}}(\varphi(z))|]\\
&\leq \frac{1}{c}\sup_{z\in D}(1-|z|^2)^\alpha |\varphi^{'}(z)| [|(f_{n})_{z}(\varphi(z))|+|(f_{n})_{\bar{z}}(\varphi(z))|]\\
&<\frac{c\varepsilon}{c}=\varepsilon.
\end{align*}
If we take $\varepsilon < \frac{1}{2}$, then we get that each $a_{n}$ with $n>N_{\varepsilon}$ belongs to $ (G_{c,\alpha})^c$. Thus $|a_{n}|\leq r_{0}<1$ and $a_{n}\rightarrow a$ with $|a|\leq r_{0}$. On the other hand, by Montel's Theorem, there exists a subsequence $\{f_{n_{k}}\}$ such that converges uniformly on compact subsets of $D$ to some function $f\in HB(\alpha)$. Hence $\{f_{n_{k}}^{'}\}$ converges to $f^{'}$  uniformly on compact subsets of $D$, and since $$\sup_{w\in G_{c,\alpha}} (1-|w|^2)^\alpha [|(f_{n})_{z}(w)|+|(f_{n})_{\bar{z}}(w)|]\rightarrow 0,$$
  when $n\rightarrow\infty$ and $G_{c,\alpha}$ contains a compact subset of $D$, we conclude that $f^{'}=0.$ This contradicts the fact that
  $$(1-|a|^2)^\alpha [|h^{'}(a)|+|g^{'}(a)|]\geq \frac{1}{2}.$$
  Therefore $C_{\varphi}$ must be bounded below and consequently it has closed range.
\end{proof}

\end{document}